%% file: arxiv-main.tex
\RequirePackage{amsmath}
\documentclass{llncs}
\usepackage{fullpage}
\usepackage{todonotes} 

\usepackage[T1]{fontenc}

\usepackage{graphicx}
\usepackage{xcolor}
\usepackage{xspace}
\usepackage{amssymb,amsfonts,graphicx}
\usepackage{algpseudocode}
\usepackage{algorithm}

\def\aiOrcid{\hspace*{.4mm}\includegraphics[scale=.5]{orcid_16x16.png}}
\definecolor{orcidlogocol}{HTML}{A6CE39}
\def\orcidID#1{\href{https://orcid.org/#1}{\textcolor{orcidlogocol}{\aiOrcid} }}

\definecolor{darkblue}{rgb}{0.0, 0.0, 0}

\newcommand{\Tau}{\scriptsize{\ensuremath{\mathcal{T}}}\normalsize\xspace}
\newcommand{\subTau}[1]{\scriptsize{\ensuremath{\mathcal{T}_{#1}}}\normalsize\xspace}

\newcommand{\cil}[1]{\ensuremath \lceil #1 \rceil} 

\newcommand{\bg}{\ensuremath{\operatorname{\texttt{BurnGuess}}}\xspace}
\newcommand{\ns}{\ensuremath{\operatorname{\texttt{no-schedule}}}\xspace}
\newcommand{\g}{g}
\newcommand{\btemp}{\ensuremath{B_{temp}}}
\newcommand{\lett}[2]{#1 $\leftarrow$ #2}

\newtheorem{observation}{Observation}

\begin{document}

\author{Shahin Kamali, Avery Miller, Kenny Zhang}
\institute{University of Manitoba}
\institute{University of Manitoba, Winnipeg, MB, Canada  \\ \ \\
\email{ \{shahin.kamali,avery.miller\}@umanitoba.ca, \\ zhangyt3@myumanitoba.ca}}

\title{Burning Two Worlds 
\subtitle{Algorithms for Burning Dense and Tree-like Graphs}}

\maketitle
\begin{abstract}
\input{abstract.tex}
\end{abstract}

\input{intro}

\input{dense}

\input{path-length}

\input{treeAlg}
\input{conclusions}

\bibliographystyle{plain}
\bibliography{refs}

\end{document}

%% file: abstract.tex
Graph burning is a simple model for the spread of social influence in networks. The objective is to measure how quickly a fire (e.g., a piece of fake news) can be spread in a network. The burning process takes place in discrete rounds. In each round, a new fire breaks out at a selected vertex and burns it. Meanwhile, the old fires extend to their adjacent vertices and burn them. A \emph{burning schedule} selects where the new fire breaks out in each round, and the \emph{burning problem} asks for a schedule that burns all vertices in a minimum number of rounds, termed the \emph{burning number} of the graph. The burning problem is known to be NP-hard even when the graph is a tree or a disjoint set of paths. For connected graphs, it has been conjectured~\cite{6BonatoJR14} that burning takes at most $\lceil \sqrt{n} \ \rceil$ rounds.  

In this paper, we approach the algorithmic study of graph burning from two directions. 
First, we consider graphs with minimum degree $\delta$. We present an algorithm that burns any graph of size $n$ in at most $\sqrt{\frac{24n}{\delta+1}}$ rounds. In particular, for dense graphs with $\delta \in \Theta(n)$, all vertices are burned in a constant number of rounds. More interestingly, even when $\delta$ is a constant that is independent of the graph size, our algorithm answers the graph-burning conjecture in the affirmative by burning the graph in at most $\lceil \sqrt{n} \rceil$ rounds.
In the second part of the paper, we consider burning graphs with bounded path-length or tree-length. These include many graph families including connected interval graphs (with path-length 1) and connected chordal graphs (with tree-length 1). 
We show that any graph with path-length $pl$ and diameter $d$ can be burned in $\lceil \sqrt{d-1} \rceil + pl$ rounds. Our algorithm ensures an approximation ratio of $1+o(1)$ for graphs of bounded path-length. 
We introduce another algorithm that achieves an approximation ratio of $2+o(1)$ for burning graphs of bounded tree-length. The approximation factor of our algorithms are improvements over the best existing approximation factor of 3 for burning general graphs.


\keywords{Graph Algorithms\and
Approximation Algorithms	 \and
Graph Burning Problem \and
\ \\ Social Contagion \and
Path-length \and
Tree-length }

%% file: intro.tex
\section{Introduction}

With the rapid growth of social networks in the past decade, numerous approaches have been proposed to study social influence in these networks \cite{BondFJKMSF12,Fajardo2013,Kramer12,Kramer14}. These studies are focused on how fast a contagion can spread in a network. A contagion can be an emotional state or a piece of data such a political opinion, a piece of fake news, a meme, or gossip. Interestingly, the spread of a contagion does not require point-to-point communication. For example, an experimental study on Facebook suggests that users can experience different emotional states as a result of being exposed to other users' posts, and this happens without direct communciation between users and even without their awareness \cite{Kramer14}. 

Given the fact that a contagion is distributed without the active involvement and awareness of users, one can argue that it is merely defined by the structure of the underlying network \cite{6BonatoJR14}.
A graph's \emph{burning number} has been suggested as a parameter that measures how prone a social network is to the spread of a contagion, which is modeled via a set of fires. Given an undirected and unweighted graph that models a social network, the fires spread in the network in synchronous rounds in the following manner. In round 1, a fire is initiated at a vertex; a vertex at which a fire is started is called an \emph{activator}. In each round that follows, two events take place. First, all existing fires spread to their neighbouring vertices, e.g., in round 2, the neighboring vertices of the first activator will be burned (i.e., they are now on fire). Second, a new fire can be started in some other part of the network, that is, a new vertex is selected as an activator at which a new fire is initiated. This process continues until a round where all vertices are on fire, at which time we say the burning `completes'. The choice of activators affects how quickly the burning process completes. A \emph{burning schedule} specifies a \emph{burning sequence} of vertices where the $i$'th vertex in the sequence is the activator in round $i$. 

The \emph{burning number} of a graph $G$, denoted by $bn(G)$, is the minimum number of rounds required to complete the burning of $G$. The graph burning problem asks for a burning schedule that completes in $bn(G)$ rounds. Unfortunately, this problem is NP-hard even for simple graphs such as trees or disjoint sets of paths \cite{4BessyBJRR17}. So, the focus of this paper is on algorithms that provide close-to-optimal solutions, that is, algorithms that burn graphs in a small (but not necessarily optimal) number of rounds.

\begin{figure}
\centering
\includegraphics[scale=.6,trim={0cm 20.1cm 5cm 17.18cm}, clip, scale=.76]{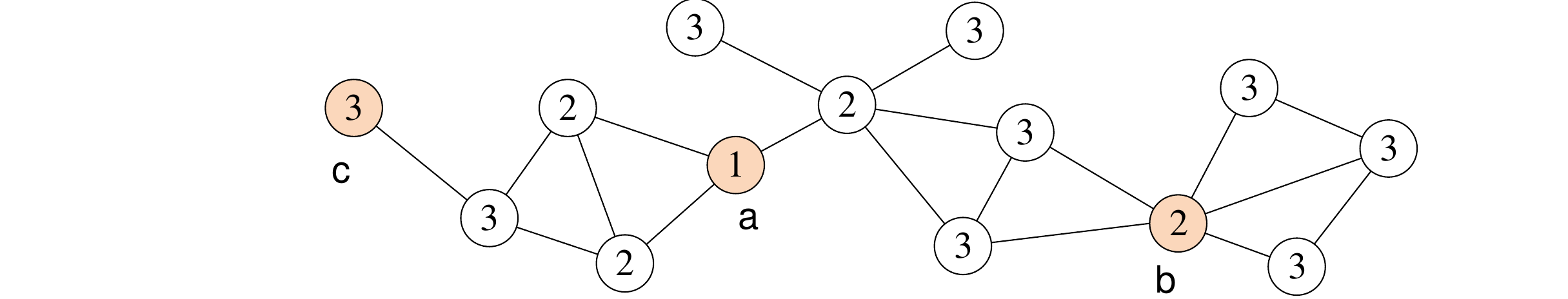}
\caption{Burning a graph using the burning schedule $\langle a, b, c\rangle$. The number at each vertex $x$ indicates the round at which the fire arrives/starts at $x$. The burning completes in 3 rounds.}
\end{figure}

\subsection*{Previous work}
The graph burning problem was introduced by Bonato et al. \cite{6BonatoJR14,5BonatoJR16} as a way to model the spread of a contagion in social networks. 
Bonato et al.~\cite{6BonatoJR14} proved that the burning number of any connected graph is at most $2\lceil\sqrt{n}\rceil-1$, and conjectured that it is always at most $\lceil \sqrt{n} \rceil$.  
Land and Lu improved the upper bound to $\frac{\sqrt{6}}{2} \sqrt{n}$ \cite{LandL16}. The conjecture, known as graph burning conjecture, is still open but verified for basic graph families~\cite{2017arXiv170709968B,DSSS18}. 
Bessy et al.~\cite{4BessyBJRR17} showed that the burning problem is NP-complete, and it remains NP-hard for simple graph families such as graphs with maximum degree three, spider graphs, and path forests. Recently, several heuristics were experimentally studied~\cite{Heuristics2019}. Bonato and Kamali~\cite{BonaKam} studied approximation algorithms for the problem. Using a simple algorithm inspired by the $k$-center problem (see, for example, \cite{Vazir0004338}), they showed that there is a polynomial time algorithm that burns any graph $G$ in at most $3bn(G)$ rounds. They also provided a 2-approximation algorithm for trees and a polynomial time approximate scheme (PTAS) for path-forests.

A line of research has been focused on characterizing the burning number for different graph families \cite{grid,5BonatoJR16,1MitschePR17,2017arXiv170709968B,product18,Paterson,ThetaGraphs19}.
This includes grids graphs~\cite{product18,grid} and more generally Cartesian product and the strong product of graphs~\cite{1MitschePR17,product18}, binomial random graphs~\cite{1MitschePR17}, random geometric graphs~\cite{1MitschePR17}, spider graphs~\cite{4BessyBJRR17,2017arXiv170709968B,DSSS18}, path-forests~\cite{4BessyBJRR17,2017arXiv170709968B,BonaKam}, generalized Petersen graphs~\cite{Paterson}, and Theta graphs~\cite{ThetaGraphs19}.

\subsection*{Contributions}
In this paper, we approach the algorithmic study of graph burning from two directions. 
In section~\ref{sect:dense}, we consider dense graphs, i.e., graphs in which we have a given lower bound $\delta$ on the minimum degree. We provide an algorithm that burns such graphs on $n$ vertices in at most $\sqrt{\frac{24n}{\delta+1}}$ rounds. Our result shows how much faster we can burn a graph with high degree, e.g., if $\delta \in \Theta(n)$, we can burn the graph in constant number of rounds. More interestingly, even when $\delta$ is a constant that is independent of the graph size, our algorithm answers the graph-burning conjecture of Bonato et al.~\cite{6BonatoJR14} in the affirmative by burning the graph in at most $\lceil \sqrt{n}\ \rceil$ rounds. 

In Section~\ref{sect:sparse}, we provide parameterized algorithms for burning graphs with small path-length and tree-length. A graph has path-length at most $pl$ (respectively tree-length at most $tl$), if there is a Robertson-Seymour path decomposition (respectively tree decomposition) of $G$ such that the distance between any two vertices located in the same bag of the decomposition is at most $pl$ (respectively $tl$). Intuitively speaking, these are graphs that can be transformed into a path (respectively tree) by contracting groups of vertices that are all at close to each other. A formal definition can be found in Section~\ref{sect:sparse}. Graphs with small path-length and tree-length span several well-known families of graphs. In particular, a connected graph is an interval graph if and only if its path-length is at most 1~\cite{gilmoreHoffman1964,phdthesisTreeLength}, and a chordal graph if and only if its tree-length is at most 1~\cite{GAVRIL74,phdthesisTreeLength}. 

We provide algorithms that burn graphs of bounded path-length and tree-length. First, we observe that if the diameter is bounded by a constant, an optimal burning schedule can be computed in polynomial time, using an exhaustive approach. So, we focus on a more interesting \emph{asymptotic setting} where the diameter of the graph is asymptotically large.
We show that any graph $G$ of diameter $d$ and path-length at most $pl$ can be burned in at most $\lceil \sqrt{d-1}\rceil + pl \leq \lceil \sqrt{n}\rceil + pl $ rounds. Since $\lceil \sqrt{d}  \rceil$ is a lower bound for the burning number of $G$, our algorithm achieves an approximation factor of 
$1+o(1)$ for graphs of bounded path-length. 
In particular, our algorithm achieves an almost-optimal solution for burning connected interval graphs. 
Moreover, we present an approximation algorithm for burning graphs of small tree-length. For a graph of tree-length at most $tl$, our algorithm has an approximation factor of at most $2+(4tl+1)/d$, which is 
$2+o(1)$ for graphs of bounded tree-length (e.g., chordal graphs). 
The approximation factors achieved by our algorithms are improvements over the best approximation ratio of 3 presented for arbitrary graphs~\cite{BonaKam}.

%% file: dense.tex
\section{Dense Graphs}\label{sect:dense}

	For any graph $G$, the \emph{degree} of a vertex is the number of edges incident to $v$. In this section, we present an algorithm that constructs a burning schedule whose length is parameterized by the minimum degree of the graph, which is defined as the minimum vertex degree taken over all of its vertices. As expected, increasing the minimum degree of the graph will decrease the number of rounds needed to burn all of the vertices, and our result sheds light on the nature of this tradeoff.
	
	An interesting consequence of our algorithm is that we make progress towards resolving the conjecture from \cite{6BonatoJR14} that every graph on $n$ vertices can be burned in at most $\lceil \sqrt{n} \rceil$ rounds. We prove that the conjecture holds for all graphs with minimum degree at least 23.
	
	To describe and analyze our algorithm, we denote by $d(v,w)$ the length of the shortest path between $v$ and $w$ in $G$, i.e., the distance between $v$ and $w$. We denote by $N_r(v)$ the set of vertices whose distance from $v$ is at most $r$. For any vertex $v \in G$, let $ecc(v)$ denote the eccentricity of $v$, i.e., the maximum distance between $v$ and any other vertex of $G$. Let $rad(G)$ denote the radius of $G$, i.e., the minimum eccentricity taken over all vertices in $G$.
	
	The idea of the algorithm is to pick a set of activators $A$ that are adequately spread apart in the graph. More specifically, for some even integer $2r$, our algorithm picks a maximal set of vertices such that the distance between any pair is greater than $2r$. This can be done efficiently on any graph $G$ in a greedy manner: pick any vertex $v$, add $v$ to $A$, remove $N_{2r}(v)$ from $G$, and repeat the above until $G$ is empty. 
	
	Since $A$ is maximal, all vertices in $G$ are within distance $2r$ from some vertex in $A$. So, if we burn one activator from $A$ in each of the first $|A|$ rounds, the fire will spread and burn all vertices within an additional $2r$ rounds (we can activate an arbitrary vertex in each of these additional rounds). The key is to pick a value for $r$ such that $|A| + 2r$ is small. This is the goal for the rest of this section.
	
	We start by finding an upper bound on $|A|$ with respect to $r$. This bound will rely on the following fact that a lower bound on the degree implies a lower bound on the size of $N_r(v)$ for any $v \in G$.
	\begin{proposition}\label{prop:BallLower}
		Consider any graph $G$ that has minimum degree $\delta$. For any vertex $v \in G$ and any $r \in \{1,\ldots,ecc(v)\}$, we have that $N_r(v) \geq \lfloor \frac{r+2}{3}\rfloor(\delta + 1)$.
	\end{proposition}
\begin{proof}
	Let $v$ be an arbitrary vertex in $G$. Define $L_i = \{w\ |\ d(v,w) = i\}$, i.e., $L_i$ is the set of vertices whose distance from $v$ in $G$ is exactly $i$. Define $S_0 = L_0 \cup L_1$, i.e., $S_0$ consists of $v$ and its neighbours. Note that $|S_0| \geq \delta+1$ since $v$ has degree at least $\delta$. Next, for each $j \in \{1,\ldots,\lfloor \frac{r+2}{3} \rfloor-1\}$, define $S_j = L_{3j-1} \cup L_{3j} \cup L_{3j+1}$. Note that, for all $j \in \{1,\ldots,\lfloor \frac{r+2}{3} \rfloor-1\}$, we have $3j+1 \leq 3(\lfloor\frac{r+2}{3} \rfloor-1) + 1 \leq r \leq ecc(v)$. The fact that $3j+1 \leq r$ means that all vertices in $L_{3j-1} \cup L_{3j} \cup L_{3j+1}$ are within distance $r$ from $v$, i.e., $S_j \in N_r(v)$ for each $j \in \{1,\ldots,\lfloor \frac{r+2}{3} \rfloor-1\}$. Moreover, the fact that $3j+1 \leq ecc(v)$ means that each of $L_{3j-1}$, $L_{3j}$, and $L_{3j+1}$ is non-empty. In particular, we can pick an arbitrary vertex in $L_{3j}$, which by assumption has at least $\delta$ neighbours, and each of these neighbours must be in one of $L_{3j-1}$, $L_{3j}$, or $L_{3j+1}$ (i.e., in $S_j$), which implies that $|S_j| \geq \delta+1$ for all $j \in \{1,\ldots,\lfloor \frac{r+2}{3} \rfloor-1\}$. Finally, by construction, $S_j \cap S_{j'} = \emptyset$ for any two distinct $j,j' \in \{0,\ldots,\lfloor \frac{r+2}{3} \rfloor-1\}$. So $|N_r(v)| \geq \sum_{j=0}^{\lfloor\frac{r+2}{3} \rfloor-1} |S_j| \geq \lfloor\frac{r+2}{3} \rfloor(\delta+1)$.\qed
\end{proof}
Next, we apply the preceding lower bound to $N_r(v)$ for each $v \in A$, and then use the fact that these neighbourhoods are disjoint to find an upper bound on $|A|$. 
	\begin{lemma}\label{lem:BBound}
		Consider any graph $G$ on $n$ vertices that has minimum degree $\delta$. Suppose that $A$ is a subset of the vertices of $G$ such that, for some $r \in \{1,\ldots,rad(G)\}$, the distance between each pair of vertices in $A$ is greater than $2r$. Then $|A| \leq \frac{3n}{r(\delta+1)}$.
	\end{lemma}
\begin{proof}
	Denote by $v_1,\ldots,v_{|A|}$ the vertices in $A$. Since the distance between each pair of these vertices is greater than $2r$, the sets $N_r(v_1),\ldots,N_r(v_{|A|})$ are disjoint, so $n \geq \sum_{i=1}^{|A|} |N_r(v_i)|$. As $r \leq rad(G)$, it follows that $r \leq ecc(v_i)$ for each $i \in \{1,\ldots,|A|\}$, so, by Proposition \ref{prop:BallLower}, we know that $|N_r(v_i)| \geq \lfloor \frac{r+2}{3}\rfloor(\delta + 1) \geq \frac{r(\delta+1)}{3}$. Therefore, $n \geq \sum_{i=1}^{|A|} |N_r(v_i)| \geq |A|\frac{r(\delta+1)}{3}$, which implies the desired result. \qed
\end{proof}
	
Recall that the burning sequence generated by our algorithm will burn the entire graph within $|A|+2r$ rounds. Finding the value of $r$ that minimizes this expression leads to the following upper bound on the length of the burning sequence.

\begin{theorem}
	For any graph $G$ on $n$ vertices that has minimum degree $\delta$, our algorithm produces a burning sequence with length at most $\left\lceil \sqrt{\frac{24n}{\delta+1}}\ \right\rceil$. 
\end{theorem}
\begin{proof}
	First, consider the case where $rad(G) < \sqrt{\frac{3n}{2(\delta+1)}}$. Then, by first activating a vertex $v \in G$ with $ecc(v) = rad(G)$, all vertices will be burned within $\sqrt{\frac{3n}{2(\delta+1)}}$ rounds. Next, consider the case where $rad(G) \geq \sqrt{\frac{3n}{2(\delta+1)}}$. Our algorithm produces a burning sequence with length at most $|A| + 2r$. Applying Lemma \ref{lem:BBound} with any $r \in \{1,\ldots,rad(G)\}$, we get that $|A| + 2r \leq \frac{3n}{r(\delta+1)} + 2r$. Notice that the value of $r$ that minimizes $\frac{3n}{r(\delta+1)} + 2r$ is $\sqrt{\frac{3n}{2(\delta+1)}}$: indeed, this can be confirmed by finding the first derivative with respect to $r$, i.e., $-\frac{3n}{r^2(\delta+1)}+2$, setting it equal to 0, and solving for $r$. Since $\sqrt{\frac{3n}{2(\delta+1)}}$ is bounded above by $rad(G)$, setting $r$ to this value implies that $|A| + 2r \leq \frac{3n}{r(\delta+1)} + 2r = \sqrt{\frac{24n}{\delta+1}}$. Combining the two cases, we have shown that our algorithm produces a burning sequence with length at most $\max\left\{\left\lceil \sqrt{\frac{24n}{\delta+1}}\ \right\rceil, \left\lceil\sqrt{\frac{3n}{2(\delta+1)}}\ \right\rceil\right\} = \left\lceil \sqrt{\frac{24n}{\delta+1}}\ \right\rceil$.\qed
\end{proof}

\begin{corollary}
	For any graph $G$ on $n$ vertices with minimum degree $\delta \geq 23$, the burning number is at most $\left\lceil \sqrt{n} \ \right\rceil$.
\end{corollary}

%% file: path-length.tex
\section{Graphs of Small Path-length or Tree-length}\label{sect:sparse}
In this section, we provide efficient algorithms for burning graphs of small path-length or tree-length. 
In Section~\ref{sect:pl:prelim}, we provide necessary definitions and preliminary results. In Sections~\ref{sect:pl:path} and~\ref{sect:pl:tree}, we provide algorithms for burning graphs of bounded path-length and tree-length, respectively. Our algorithms achieve good approximation ratios when the diameter of the input graph is asymptotically large. We note that the burning problem is more interesting when the graphs has large diameter; otherwise, we can use the following easy theorem that shows the problem can be optimally solved 
for graphs of small diameter.

\begin{theorem}\label{th:smallDiam}
The burning problem can be optimally solved in polynomial time if the diameter of the input graph is bounded by a constant.
\end{theorem}

\begin{proof}
Assume that the diameter an input graph $G$ is a constant value $c$. We know the burning number of $G$ is at most $c$. Any burning schedule can be described by a sequence of activators $\langle a_1, a_2, \ldots, a_{\gamma} \rangle$ for some positive integer $\gamma \leq c$ which defines the cost of the schedule. Given a fixed value of $\gamma$, there are at most 
$n^{\gamma} \leq n^c$ possible ways to select the $\gamma$ activators. 
So, there are less than $c \times n^c$ candidate burning schedules. For each candidate, we can check whether all vertices are burned in $\gamma$ rounds using a breadth-first approach, in $O(n^2)$ time. So, it takes $O(c \times n^{c+2})$ to find an optimal burning schedule, that is, a candidate solution which burns all vertices and has smallest burning time (the smallest value of $\gamma$).  \qed
\end{proof}

In light of Theorem~\ref{th:smallDiam}, for the remainder of this section, we focus on graphs with asymptotically large diameter.

\subsection{Preliminaries} \label{sect:pl:prelim}
The concepts of \emph{path decomposition} and \emph{tree decomposition} \cite{Treewidth1,Treewidth2}, 
were initially intended to measure, via the path-width and tree-width parameters, how close a graph is to a path and a tree, respectively. Path-length and tree-length are related parameters that are also based on the same definition of path decomposition.

\begin{definition}
	[Decompositions, tree-length, path-length]
	\ 
	\begin{itemize}
\item A \emph{tree decomposition} \Tau of a graph $G$ is a tree whose vertex set is a finite set of \emph{bags} $\{B_i\ |\ 1 \leq i \leq \xi \in \mathbb{N}\}$, where: 
each bag is a subset of the vertices of $G$; for every edge $\{v,w\}$, at least one bag contains both $v$ and $w$; and, for every vertex $v$ of $G$, the set of bags containing $v$ forms a connected subtree of \Tau. When \Tau is a path, then the decomposition is called a \emph{path decomposition} of $G$. 
\item A \emph{rooted tree decomposition} is a tree decomposition with a designated root bag, and parent/child relationships between bags are defined in the usual way. For any bag $B$ in a rooted decomposition \Tau, we denote by \subTau{B} the subtree of the decomposition rooted at $B$.
\item The \emph{length} of a decomposition is the maximum distance between two vertices in the same bag, i.e.,
$\max_{1\leq i \leq \xi} \{d(x,y)\ |\ x,y \in B_i\}$.  
The \emph{tree-length of $G$}, denoted by $tl$, is defined to be the minimum length taken over all tree decompositions of $G$. The \emph{path-length of $G$}, denoted by $pl$, is defined to be the minimum length taken over all path decompositions of $G$. 
\end{itemize}
\end{definition}

Figure~\ref{fig:treeDecompos} illustrates the concepts of path-length and tree-length. Throughout the paper, we refer to vertices of \Tau as bags to distinguish them from vertices of $G$. When discussing graphs of small path/tree-length, we implicitly assume the input graph is connected (otherwise, its path/tree-length would be unbounded). 
For any graph $G$, the path-length of $G$ cannot be smaller than its tree-length, that is, the family of graphs with bounded tree-length includes graphs with bounded path-length as a sub-family.
It is known that a graph has path-length 1 if and only if it is a connected interval graph~\cite{gilmoreHoffman1964}, and tree-length 1 if and only if it is a connected chordal graph~\cite{GAVRIL74}. As such, it is possible to recognize and compute the path/tree decomposition of these graphs in linear time:  we can use the linear-time algorithm of Booth and Lueke~\cite{BoothL76} for interval graphs and a lexicographic breadth first search~\cite{RoseTL76} for chordal graphs. Unfortunately however, we cannot extend these algorithms to larger values of path-length or tree-length. In particular, it is known that the problem of determining whether a given graph has tree-length at most $k$ is NP-hard for any $k\geq 2$~\cite{Lokshtanov2010}. On the positive side, there are algorithms with approximation factor 2 for computing path-length~\cite{phdthesisTreeLength}, and approximation factor 3 for computing tree-length\footnote{In contrast, it is pretty hard to approximate the tree-width; see, e.g.,~\cite{CoudertDN16} for details.}~\cite{DourisboureG07}. 
Given these results, it is safe to assume a path/tree decomposition of a given graph is provided together with the graph (otherwise, we use these algorithms to achieve decompositions which are a constant factor away from the optimal decomposition).

%
%

\begin{figure}
\begin{center}
        \includegraphics[scale=.5, trim= 0mm 130mm 20mm 0mm , clip]{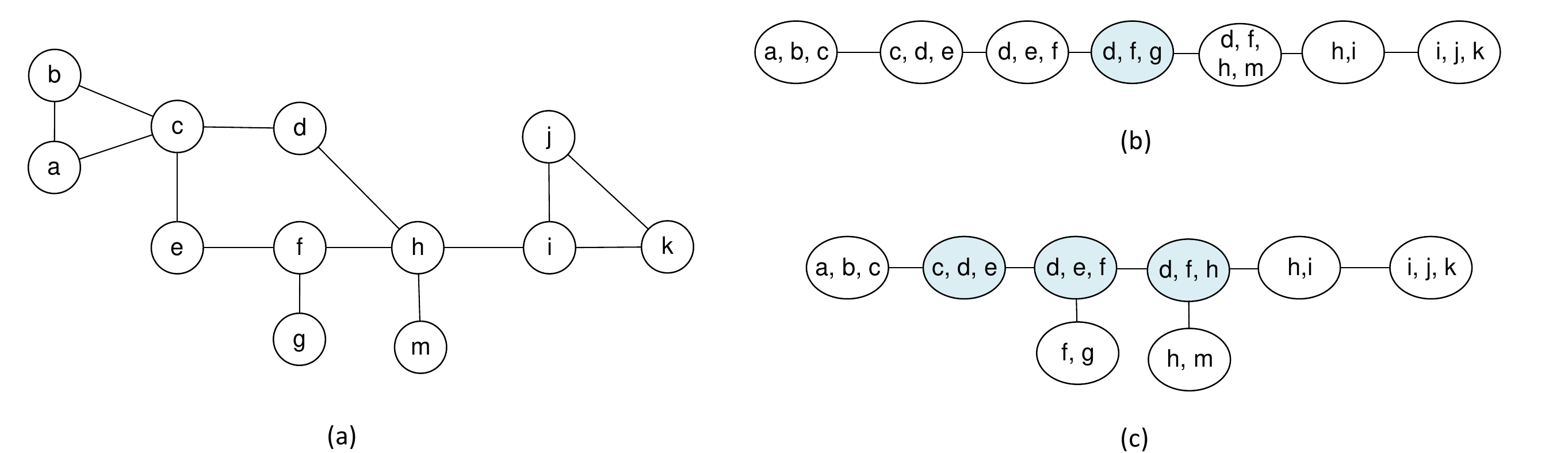}
\end{center}
    \caption{(a) A graph $G$  (b) A path decomposition of $G$ with path-length 3. Note that the distance between $d$ and $g$ in the highlighted bag is 3. (c) A tree decomposition of $G$ with tree-length 2. The highlighted bags contain pair of vertices with distance 2.} 
    \label{fig:treeDecompos}

\end{figure}


The following results establish useful structural properties about tree decompositions that will be used in the remainder of the paper. The first observation is that, in any rooted decomposition \Tau, if there is a bag $B$ such that the set of all vertices in bags of \subTau{B} is a subset of $B$'s parent bag, then we can remove the entire subtree rooted at $B$ from \Tau without changing the length of the decomposition. After removing all such subtrees, we call the remaining decomposition ``trimmed''.
\begin{definition}
	For any connected graph $G$, a rooted tree decomposition \Tau of $G$ is called \emph{trimmed} if, for each bag $B$ of \Tau other than the root, there exists a vertex $v \in G$ that is contained in a bag of \subTau{B} and not contained in any bag of \Tau $\setminus$ \subTau{B}.
\end{definition}
\begin{observation}\label{obs:nonTrivial}
	For any connected graph $G$ and any rooted tree decomposition \scriptsize{$\mathcal{T}'$}\normalsize\xspace, there exists a trimmed rooted tree decomposition \Tau of $G$ with the same length as \scriptsize{$\mathcal{T}'$}\normalsize\xspace.
\end{observation}

In the remainder of the paper, we assume that all rooted tree decompositions are trimmed.
\begin{lemma}\label{lem:treeLength}
	Consider any rooted tree decomposition \Tau $= \{B_i\ |\ 1 \leq i \leq \xi \in \mathbb{N}\}$ of a graph $G$, and denote the root bag by $B_r$. For any bag $B_j \neq B_r$, define $G_j$ to be the graph obtained by removing from $G$ all vertices contained in bag $B_j$. The vertices in bags of \subTau{B_j} and the vertices in bags of \Tau $\setminus$ \subTau{B_j} are in different components of $G_j$, i.e., $G_j$ is disconnected.
\end{lemma}

\begin{proof}
	Consider any $j \in \{1,\ldots,\xi\} \setminus \{r\}$, and define $G_{j}$ as the graph obtained by removing from $G$ all vertices contained in bag $B_j$. To obtain a contradiction, assume the vertices in bags of \subTau{B_j} and the vertices in bags of \Tau $\setminus$ \subTau{B_j} are in the same component of $G_j$.
	Since \Tau is trimmed, there exists an edge $e=\{x,y\}$ in $G_j$ (and in $G$) such that $x$ is contained in the subgraph induced by the vertices in bags of \subTau{B_j}, and $y$ is contained in the subgraph induced by the vertices in bags \Tau $\setminus$ \subTau{B_j}. As vertices $x$ and $y$ are in $G_j$, they are not contained in bag $B_j$ of \Tau. Since \Tau is a valid path decomposition, the set of bags containing $x$ must form a connected subpath in \Tau, and the set of bags containing $y$ must form a connected subpath in \Tau. It follows that $x$ cannot appear in any bag of \Tau $\setminus$ \subTau{B_j} and $y$ cannot appear in any bag of \subTau{B_j}. So, $x$ and $y$ do not appear together in any bag of \Tau, despite being adjacent in $G$, which contradicts the fact that \Tau is a valid path decomposition of $G$. \qed
\end{proof}

We immediately get the following result concerning path decompositions by choosing either leaf bag as the root and applying Lemma \ref{lem:treeLength}.
\begin{corollary}\label{cor:pathLength}
	Consider any path decomposition \Tau $= \{B_i\ |\ 1 \leq i \leq \xi \in \mathbb{N}\}$ of a graph $G$ with bags indexed in order from one leaf node to the other. For any $j \in \{2,\ldots,\xi-1\}$, define $G_j$ to be the graph obtained by removing from $G$ all vertices contained in bag $B_j$. The vertices in bags $B_1,\ldots,B_{j-1}$ and the vertices in bags of $B_{j+1},\ldots,B_\xi$ are in different components of $G_j$, i.e., $G_j$ is disconnected.
\end{corollary}

\begin{lemma}\label{lem:crucialBags}
Consider any rooted tree decomposition or any path decomposition \Tau $= \{B_i\ |\ 1 \leq i \leq \xi \in N\}$ of a graph $G$. Let $u$ and $v$ be any vertices in $G$, let $B_u$ be any bag of \Tau that contains $u$, and let $B_v$ be any bag of \Tau that contains $v$. If $P$ is a shortest path between $u$ and $v$ in $G$, then each bag in the shortest path between $B_u$ and $B_v$ in \Tau contains a vertex of $P$.
\end{lemma}
\begin{proof}
To obtain a contradiction, assume that there exists a bag $B_j$ in the shortest path between $B_u$ and $B_v$ in \Tau that does not contain any vertex of $P$. Since $u \in B_u$ and $v \in B_v$, it follows that $B_j \notin \{B_u, B_v\}$. Since $B_j$ does not contain any vertex from $P$, the path $P$ remains unchanged in the subgraph $G_j$ of $G$ formed by removing all vertices contained in $B_j$. 
By Lemma~\ref{lem:treeLength} (if \Tau is a rooted tree decomposition) and Corollary \ref{cor:pathLength} (if \Tau is a path decomposition), $G_j$ is disconnected and $u$ and $v$ are located in different components of $G_j$. This means $P$ is a path between $x$ and $y$ that belong to different components of $G_j$, which is a contradiction.	\qed
\end{proof}

\subsection{Burning graphs of small path-length}\label{sect:pl:path}
The following theorem shows that a graph of bounded path-length can be burned in a nearly-optimal number of rounds.

\begin{theorem}\label{th:pl}
Given a graph $G$ of diameter $d \geq 1$ and a path decomposition of $G$ with path-length $pl$, it is possible to burn $G$ in $\lceil \sqrt{d-1}\ \rceil + pl$ rounds.
\end{theorem}

\begin{proof}
Consider a path decomposition \Tau $= \{B_i\ |\ 1 \leq i \leq \xi \in N\}$ of $G$ such that the bags are indexed in increasing order from one leaf to the root. Further, assume that \Tau has the following form: the first bag $B_1$ contains a vertex $x$ that is absent in $B_2$, and the last bag $B_\xi$ contains a vertex $y$ that is absent in bag $B_{\xi-1}$. For any path decomposition of $G$ that is not of this form, it must be the case that at least one of $B_1 \subseteq B_2$ or $B_{\xi} \subseteq B_{\xi-1}$ holds. If $B_1 \subseteq B_2$, we can remove $B_1$ to get another path decomposition of $G$, and if $B_{\xi} \subseteq B_{\xi-1}$, we can remove $B_\xi$ to get another path decomposition of $G$. If \Tau consists of one bag $B_1$, then the diameter of $G$ is $pl$, so $G$ can be burned within $pl$ rounds by choosing any vertices of $G$ as activators. So we proceed under the assumption that $\xi \geq 2$.

Since \Tau is a valid path decomposition, each neighbour $x'$ of $x$ must appear together with $x$ in at least one bag, and this must be $B_1$: as we assumed that $x$ is in $B_1$ and not $B_2$, we know that $x$ does not appear in any bag $B_i$ with $i \geq 2$ as the bags containing $x$ must form a connected subgraph of \Tau. Similarly, each neighbour $y'$ of $y$ must appear together with $y$ in $B_{\xi}$. Thus, the shortest path between $x$ and $y$ in $G$ starts with an edge $\{x,x'\}$ such that $x'\in B_1$ and ends with an edge $\{y',y\}$ such that $y' \in B_{\xi}$.  
Let $S$ denote the shortest path between $x'$ and $y'$ in $G$; note that $S$ has length at most $d-2$. It is known that any path of length $m$ can be burned in $\lceil \sqrt{m+1} \rceil$ rounds~\cite{5BonatoJR16}. So, we can devise a burning schedule that burns all vertices of $S$ in at most $\lceil \sqrt{d-1} \rceil$ rounds.  
By Lemma~\ref{lem:crucialBags}, each bag in $\{B_2, B_3, \ldots, B_{\xi-1}\}$ contains at least one vertex of $S$. Recall that $x' \in S$ is in $B_1$ and $y' \in S$ is in $B_{\xi}$. So, within $\lceil \sqrt{d-1} \rceil$ rounds that it takes to burn all vertices of $S$, at least one vertex in each bag of the path decomposition is burned. In the $pl$ rounds that follow, all vertices will be burned because the distance between any two vertices in each bag is at most $pl$ (from the definition of path-length). \qed  
\end{proof}

The study of path-length is relatively new, and its relationship with other graph families is not fully discovered yet. Regardless, we can still use Theorem~\ref{th:pl} to state the following two corollaries about grids and interval graphs.

\begin{corollary}
Consider a grid graph $G$ of size $n = n_1 \times n_2$ and $n_1 \leq n_2$. It is possible to burn $G$ in $\sqrt{n} + o(\sqrt{n})$ rounds.
\end{corollary}

\begin{proof}
Without loss of generality, assume $G$ is formed by $n_1$ rows and $n_2$ columns.
Consider a path decomposition $P= \{B_1, \ldots, B_{n_2-1}\}$, where $B_i$ includes vertices in columns $i$ and $i+1$ of $G$. The length of this decomposition is $n_1$. Now, if $n_1 \in \Theta(1)$, the burning time devised by Theorem~\ref{th:pl} is $\sqrt{d-1} + pl < \sqrt{n} + \Theta(1)$. If $n_1 \in \omega(1)$, the diameter $d$ will be asymptotically smaller than $n$, and the burning time will be $\sqrt{d-1} + pl < o(\sqrt{n}) + n_1 \leq o(\sqrt{n}) + \sqrt{n}$. \qed 
\end{proof}


\begin{corollary}\label{coro:plength}
Any interval graph $G$ of diameter $d$ and size $n$ can be burned within $\lceil \sqrt{d-1}\ \rceil + 1 \leq \lceil \sqrt{n}\ \rceil + 1$ rounds. 
\end{corollary}

\begin{proof}
There is an intuitive way to form a path decomposition of interval graphs with length 1. For that, one can scan intervals from left to right. Whenever a new interval $I$ is started, a new bag with vertices associated with intervals present at the starting point of $I$ is added. Similarly, when an interval $I'$ is ended, a new bag is added where the vertex associated with $I'$ is removed. It is easy to verify that this path decomposition is a valid decomposition of the graph. Vertices in each bag form a clique and hence the decomposition has path-length 1. So, by Theorem~\ref{th:pl}, it is possible to burn the graph in $\lceil \sqrt{d-1}\ \rceil +1$ rounds. \qed
\end{proof}

Finally, we show that the algorithm used to prove Theorem~\ref{th:pl} guarantees a $1+o(1)$-approximation factor.

\begin{corollary} 
Given any graph $G$ of bounded path-length, there is an algorithm with approximation factor $1+o(1)$. 
\end{corollary}

\begin{proof}
First, if the diameter $d$ of $G$ is bounded by a constant, we use Theorem~\ref{th:smallDiam} to optimally burn $G$. Next, assume $G$ has asymptotically large diameter. In case we are not provided with a path decomposition of $G$, we use the algorithm of~\cite{phdthesisTreeLength} to achieve a path decomposition with a length $pl$ that is at most twice the path-length of $G$ (and hence is bounded). Given the path decomposition, we apply Theorem~\ref{th:pl} to burn $G$ in $\lceil \sqrt{d-1}\ \rceil + pl$ rounds. An optimal burning schedule requires at least $\lceil \sqrt{d+1}\ \rceil$ rounds to complete the burning process~\cite{5BonatoJR16}. So, our algorithm achieves an approximation ratio of $\frac{\lceil \sqrt{d-1} \rceil + pl}{\lceil \sqrt{d+1} \rceil} < 1 + pl/\sqrt{d}$, which is $1+o(1)$ as $pl$ is bounded by a constant and $d$ is asymptotically large. \qed
\end{proof}

%% file: treeAlg.tex
\subsection{Burning graphs of small tree-length}\label{sect:pl:tree}

In this section, we consider the burning problem in connected graphs of bounded tree-length. These graphs include trees as a subfamily (trees are chordal and hence have tree-length 1).
We note that trees are the ``hardest'' connected graphs to burn 
in the sense that if we can burn any tree of size $n$ in $f(n)$ rounds, then we can burn any graph $G$ of size $n$ in the same number of rounds (just take a spanning tree of $G$ and burn it in $f(n)$ rounds). 
In this section, instead of bounding the burning time as a function of $n$, we introduce an approximation algorithm for burning graphs of bounded tree-length. Bounding the approximation ratio might be more meaningful than experessing the burning time as a function of $n$ in the sense that many instances of trees can be burned in less than $\cil{\sqrt{n}}$ rounds. As an example, every star tree can be burned in 2 rounds. For trees, there is a simple algorithm with approximation factor 2~\cite{BonaKam}. Our algorithm in this section can be seen as an extension of this algorithm to graphs of bounded tree-length.

In what follows, we assume a tree decomposition \Tau of a graph $G$ is given; otherwise, we use the algorithm of~\cite{DourisboureG07} to obtain a 3-approximation of the tree-length. We use $tl$ to denote the length of \Tau. 
We assume that the tree \Tau is rooted at an arbitrary bag, and an arbitrary vertex in the root bag is called the ``origin'' vertex and denoted by $o$.

Our algorithm is based on a procedure named \bg that receives a graph $G$ and a positive integer $\g$ as its inputs. The output of \bg is either I) \ns, which means that there does not exist a burning schedule such that the burning process completes within $\g$ rounds, or, II) a burning schedule such that all vertices are burned within $2\g+4tl+1$ rounds. Let $\g^*$ denote the smallest value of $\g$ for which a schedule is returned by \bg. Since \bg returns \ns for $\g^*-1$, it follows that an optimal schedule requires at least $\g^*$ rounds to burn the graph, while the schedule returned by \bg for $\g^*$ burns the graph within $2\g^*+4tl+1$ rounds. The approximation factor of such a schedule is consequently $\frac{2\g^*+4tl+1}{\g^*} = 2 + (4tl+1)/\g^*$ which is $2+o(1)$ assuming $G$ has asymptotically large diameter and bounded tree-length.

Procedure \bg works by marking the vertices of $G$ in iterations. Initially, no vertex is marked. In what follows, we describe the actions of the algorithm in iteration $i \geq 1$. At the beginning of iteration $i$, an arbitrary unmarked vertex at maximum distance in $G$ from the origin $o$ is selected and called \emph{terminal} $t_i$.
Let $B_i$ be a bag of \Tau with minimum depth (distance from the root) that contains $t_i$. We traverse $B_i$ towards the root until we find a bag $B_i'$ such that all vertices in $B'_i$ are at distance at least $g$ from $t_i$ in $G$. If there is no such $B'_i$, the root of \Tau is chosen as $B'_i$. 
We select an arbitrary vertex in $B'_i$ as the $i^{th}$ activator and denote it with $c_i$.
After selecting $c_i$, all vertices that are within distance $(2\g-i+1) +4tl$ from $c_i$ are marked, and this ends iteration $i$. The above process continues until all vertices in \Tau are marked or when the number of iterations exceeds $\g$. 
Algorithm~\ref{Alg:bg} describes the \bg procedure in detail, and Figure~\ref{fig:exAlg} illustrates the algorithm.

\begin{algorithm}[!t]
\begin{algorithmic}
\State \textbf{Input:} A connected graph $G=(V,E)$, a tree decomposition $\mathcal{T}$ of $G$ with tree-length $tl$, an integer $\g \geq 1$. 
\State \textbf{Output:} Either ``\ns'' or a burning schedule $\langle c_1,c_2,\ldots,c_m \rangle$ with $m\leq \g+1$. 
\State \lett{$B_r$}{an arbitrary bag in $\mathcal{T}$}; \hspace*{.1cm} assume $\mathcal{T}$ is rooted at $B_r$ 
\State \lett{$o$}{an arbitrary vertex in $\mathcal{T}$;} \lett{$i$}{$0$}; \lett{Marked}{$\emptyset$}
\While{ $i \leq \g$ and Marked $\neq V$}
\State $i \gets i+1$ 
\State \lett{$t_i$}{argmax$_u \ d(u,o)$ such that $u \notin$ Marked}
\State \lett{$B_i$}{argmin$_C \	 d(C,B_r)$ such that $t_i \in C$}
\State \lett{\btemp}{$B_i$}
\While{$\exists \ v\in V$ such that $v \in \btemp$ \textbf{and} $d(v,t_i) \leq g -1$ \textbf{and} $\btemp \neq B_r$}
\State \lett{\btemp}{parent(\btemp)}
\EndWhile
\State \lett{$B'_i$}{\btemp}
\State $c_i\gets$ any vertex in $B'_i$
\ForAll {$u \in V$ such that $d(u,c_i) \leq (2\g - i + 1) + 4tl$}
\State Marked $\gets$ Marked $\cup \ \{u\}$ 
\EndFor
\EndWhile
\If {Marked == $V$}
    \State \textbf{return} $\langle c_1,c_2,\ldots, c_i\rangle$.
\Else
	\State \textbf{return} \ns
\EndIf
\end{algorithmic}
\caption{\bg Procedure} \label{Alg:bg}
\end{algorithm}

\begin{figure}
\begin{center}
\includegraphics[scale=.4, trim = 0mm 446mm 55mm 0mm , clip]{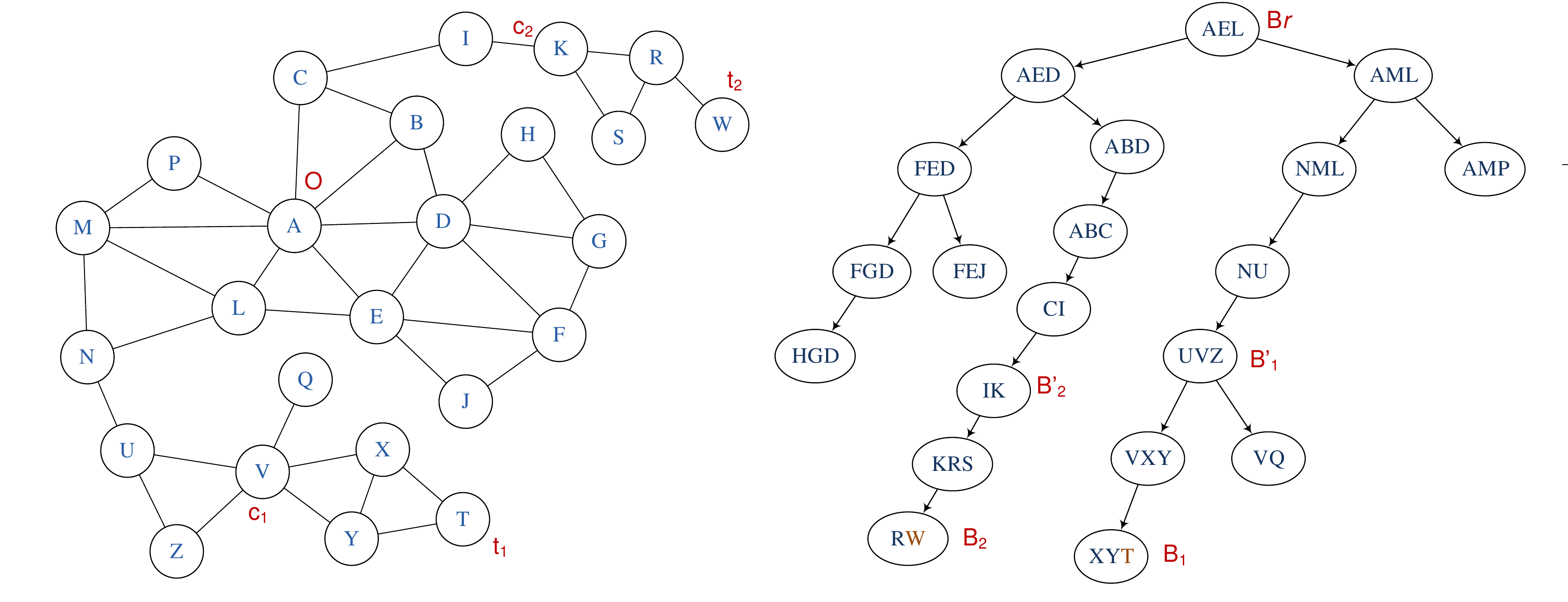}
\end{center}
\caption{\footnotesize An illustration of the \bg procedure with parameter $g=2$ executed on a chordal graph $G$ (left). On the right, a tree decomposition of tree-length $tl=1$ is rooted at arbitrary bag $B_r=\{A,E,L\}$; an arbitrary vertex $A\in B_r$ is selected as the origin $o$. In iteration $i=1$, the furthest unmarked vertex from $A$ is vertex $T$; so $t_1=T$ and $B_1=\{X,Y,T\}$. Traversing towards the root from $B_1$, the first bag in which all vertices have distance at least $g=2$ from $T$ in $G$ is $B'_1=\{U,V,Z\}$. An arbitrary vertex from $B'_1$ is selected as an activator, say, $c_1=V$. All vertices at distance $(2g-i+1)+4tl = 8$ from $V$ are marked. The only unmarked vertex would be $W$, which will be selected as $t_2$ in the next iteration.  }\label{fig:exAlg}
\end{figure}

To prove the desired upper bound on the approximation factor of our algorithm, first we establish the following lemma that provides an upper bound on the burning time of any burning schedule returned by \bg.

\begin{lemma}\label{lem:upperBoundTL}
If \bg on input $G$ returns a burning schedule $A$, then the burning process corresponding to schedule $A$ completes within $2\g + 4tl +1$ rounds.
\end{lemma}
\begin{proof}
The procedure returns a schedule only if all vertices are marked within $\g+1$ iterations. For each vertex $u$, consider the iteration $i$ in which $u$ is marked. From the definition of \bg, it follows that $d(u,c_i) \leq (2\g - i + 1) + 4tl$. Burning the graph according to schedule $A$, activator $c_i$ is burned in or before round $i$. So, vertex $u$ is burned no later than round $i + d(u,c_i) \leq 2\g + 4tl +1$. \qed 
\end{proof}

Our next goal is to establish a lower bound for the burning number of graph $G$ in the case that \bg returns \ns on input $G$. To this end, we first prove the following useful lemma.

\begin{lemma}\label{lem:distB}
For any connected graph $G$, after each iteration $i$ of \bg on input $G$, each vertex in $B'_i$ is within distance $\g+2tl - 1$ of $t_i$ in $G$. 
\end{lemma}

\begin{proof}
First, consider the case where $B'_i = B_i$. Since $t_i \in B_i$, the definition of tree-length implies that each vertex in $B'_i$ is within distance $tl \leq \g+2tl$ of $t_i$ in $G$, as desired. Next, suppose that $B'_i \neq B_i$, and consider an arbitrary vertex $u \in B'_i$. Let $C$ be the child of $B'_i$ on the path between $B_i$ and $B'_i$. Since $B'_i$ is the first ancestor of $B_i$ in which all vertices are at distance at least $g$ from $t_i$ in $G$, we conclude that there is a vertex $v \in C$ that is within distance $g-1$ from $t_i$ in $G$, that is $d(v,t_i) \leq g-1$. 

\vspace{1mm}
\textbf{Claim:} There exists a vertex $w \in G$ such that $w \in B'_i$ and $w \in C$.
\vspace{1mm}

To prove the claim, assume otherwise, i.e., $B'_i \cap C = \emptyset$. Denote by \subTau{C} the subtree of the tree decomposition rooted at bag $C$. As \Tau is a valid tree decomposition, for any choice of vertex $x \in V(G)$, the set of bags containing $x$ forms a connected component. Therefore, $B'_i \cap C = \emptyset$ implies that each $x \in V(G)$ that is contained in a bag of \subTau{C} is not contained in any bag of \Tau $\setminus$ \subTau{C}. Similarly, each $y \in V(G)$ that is contained in a bag of \Tau $\setminus$ \subTau{C} is not contained in any bag of \subTau{C}. Thus, we have shown that for each $x \in C$ and $y \in B'_i$, the vertices $x,y$ do not appear together in any bag of \Tau. As \Tau is a valid tree decomposition, it follows that there is no edge between the set of vertices induced by the bags of \subTau{C} and the set of vertices induced by the bags of \Tau $\setminus$ \subTau{C}, which implies that $G$ is not connected, a contradiction. This completes the proof of the claim.

Let $w \in B'_i \cap C$. Since $u$ and $w$ both appear in $B'_i$, we have $d(u,w)\leq tl$. Similarly, since $w$ and $v$ both appear in $C$ and $d(v,w) \leq tl$. Therefore, the distance between $u$ and $t_i$ is $d(u,t_i) \leq d(u,w) + d(w,v) + d(v,t_i) \leq 2tl + \g - 1$. \qed
\end{proof}

\begin{lemma}\label{lem:lowerBoundTL}
If the \bg procedure returns \ns for inputs $G,g$, then there is no burning schedule such that the corresponding burning process burns all vertices of $G$ in fewer than $\g$ rounds.
\end{lemma}

\begin{proof}
From the definition of \bg, the value \ns is returned when $\g+1$ iterations been completed and there exists an unmarked vertex in $G$. 
For any iteration $i$, Lemma~\ref{lem:distB} ensures that the distance between $c_i$ and $t_i$ is at most $\g+2tl-1$. In iteration $i$, all vertices within distance $(2\g-i+1)+4tl \geq \g+2tl-1$ from $c_i$ in $G$ are marked, so vertex $t_i$ is marked by the end of iteration $i$. As a result, the $g+1$ iterations involve $g+1$ different terminal vertices $t_1,\ldots,t_{g+1}$.

Let $S$ be the set consisting of the terminal vertices $t_1,\ldots,t_{g+1}$, excluding the terminal $t_k$ whose corresponding activator $c_k$ is located in the root bag $B_r$ of \Tau, if such a terminal exists. Note that there are at least $g$ vertices in $S$. The following claim gives a useful fact about the bags that contain terminals from $S$.

\vspace{1mm}
\textbf{Claim 1:} For any $t_i, t_j \in S$ such that $i < j$, terminal $t_j$ appears in a bag of \Tau $\setminus$ \subTau{B'_i}.
\vspace{1mm}

To prove the claim, we assume, for the sake of contradiction, that all bags containing $t_j$ appear in the subtree of \Tau rooted at $B'_i$. 
Since the bags that contain $t_j$ are in the subtree rooted at $B'_i$, by Lemma~\ref{lem:crucialBags}, the shortest path between $t_j$ and the origin $o$ passes through a vertex $x_j \in B'_i$, that is $d(t_j,o) = d(t_j,x_j)+d(x_j,o)$. 
%
Since $t_i$ has the maximum distance from the origin among unmarked vertices when it is selected as the $i^{th}$ terminal, we have $d(t_i,o) \geq d(t_j,o)$. We can write: 


%

\begin{align}
d(t_i,o) \geq d(t_j,o) & \Leftrightarrow  d(t_i,c_i)+d(c_i,o) \geq d(t_j,x_j)+d(x_j,o) \label{line1}  \\ & \Leftrightarrow d(t_i,c_i)+d(c_i,o) \geq (d(t_j,c_i)-d(c_i,x_j))+(d(c_i,o)-d(x_j,c_i)) \label{line2}  \\ & \Leftrightarrow d(t_i,c_i) \geq  d(t_j,c_i)-2d(c_i,x_j))\nonumber  \\ & \Leftrightarrow d(t_i,c_i) \geq  d(t_j,c_i)-2tl \label{line3} \\ & \Leftrightarrow  \g + 2tl - 1 \geq d(t_j,c_i)-2tl  \label{line4}
\end{align}

Inequalities~\ref{line1} and~\ref{line2} follow from triangle inequality and the fact that $x_j$ is on the shortest path between $t_j$ and $o$. Inequality~\ref{line3} holds because $c_i$ and $x_j$ are both in $B'_i$ and hence their distance is at most $tl$. Inequality~\ref{line4} follows from Lemma~\ref{lem:distB}. From Inequality~\ref{line4}, we conclude that $t_j$ is within distance $\g + 4tl - 1 \leq (2g-i+1) + 4tl$ from $c_i$, that is, $t_j$ should have been marked at the end of iteration $i$; this contradicts that $t_j$ was chosen as a terminal from among unmarked vertices in iteration $j > i$. This completes the proof of Claim 1.

Next, using Claim 1, we prove that any two terminals in $S$ are far apart in $G$.

\vspace{1mm}
\textbf{Claim 2:} The pairwise distance between any two terminals in $S$ is at least $2\g$ in $G$.  \\
\vspace{1mm}
Consider any two terminals $t_i,t_j$ in $S$, and assume, without loss of generality, that $i < j$. As $B'_i$ is an ancestor of $B_i$, Claim 1 implies that the shortest path in \Tau between bags $B_i$ and $B_j$ passes through $B'_i$.  
By Lemma~\ref{lem:crucialBags}, it follows that $B'_i$ contains a vertex $x$ on the shortest path between $t_i$ and $t_j$ in $G$. We can write 
\begin{align*}
d(t_i,t_j) & = d(t_i,x) + d(t_j,x) \\ & \geq (d(t_i,c_i) - d(x,c_i)) + (d(t_j,c_i)-d(x,c_i)) \\ & = d(t_i,c_i) + d(t_j,c_i) -2d(x,c_i) \\ & \geq  d(t_i,c_i) + d(t_j,c_i) - 2tl
\end{align*} 
The first two inequalities follow directly from the triangle inequality, and the last inequality holds because $x$ and $c_i$ both appear in $B'_i$ and we have $d(x,c_i) \leq tl$. 
By definition, all vertices in $B'_i$, and in particular $c_i$, are at distance at least $g$ from $t_i$, that is $d(t_i,c_i)\geq g$. 
Moreover, since terminals are chosen from among unmarked vertices, and $t_j$ was chosen as a terminal in iteration $j>i$, it follows that $t_j$ was not marked at the end of iteration $i$, which implies that $d(t_j,c_i) > (2g - i + 1) + 4tl \geq g +4tl$. We can conclude $d(t_i,t_j) > g + (g +4tl) - 2tl >2g$. This completes the proof of Claim 2.

Using Claim 2, we are able to prove the statement of the lemma. Assume, for the sake of contradiction, that there is a burning schedule $A$ that burns all vertices of $G$ in $t \leq g-1$ rounds. In particular, this means that $A$ specifies at most $t \leq g-1$ activators. In a burning process that lasts $t$ rounds, a vertex $v \in V(G)$ is burned only if $v$ is within distance $g-1$ from at least one activator. Thus, if every vertex of $G$ is burned, then each terminal in $S$ is burned, so each terminal is within distance $g-1$ from at least one activator. However, no two terminals $t_i,t_j \in S$ can be within distance $g-1$ from the same activator, since, by Claim 2, $d(t_i,t_j)$ is at least $2g$. This implies that $A$ contains at least $|S| = g$ activators, a contradiction. Hence, any burning schedule requires at least $g$ rounds, which completes the proof.  \qed

\end{proof}

\begin{theorem}
Given a graph $G$, and a tree decomposition $tl$ of $G$, there is a polynomial-time algorithm for burning $G$ that achieves an approximation factor of $2+o(1)$. 
\end{theorem}

\begin{proof}
In case the diameter of $G$ is bounded, apply Theorem~\ref{th:smallDiam} to burn it optimally in polynomial time. Otherwise, repeatedly apply the \bg procedure to find the smallest parameter $g^*$ for which the algorithm returns a schedule $A$. By Lemma~\ref{lem:upperBoundTL}, the burning process corresponding to schedule $A$ completes within $2g^* + 4tl +1$ rounds. Meanwhile, by Lemma~\ref{lem:lowerBoundTL}, no burning schedule can result in burning all vertices of $G$ within $g^*-1$ rounds. The approximation ratio of the algorithm will be $\frac{2g^*+4tl+1}{g^*} = 2+(4tl+1)/g^*$, which is $2+o(1)$ as the diameter of $G$, and hence $g^*$, is asymptotically large. 

With respect to the time complexity, finding $g^*$ requires a binary search in the range $g\in [1,d]$, that is, calling \bg $O(\log d)$ times (where $d$ is the diameter of $G$). Meanwhile, each call to \bg with parameter $g$ has at most $g \leq d$ iterations. Iteration $i$ involves finding the unmarked vertex $t_i$ with maximum distance from the origin; this can be done in $O(n^2)$ time. Finding $B'_i$ and declaring its center $c_i$ also takes $O(n^2$); this can be done using Dijkstra's algorithm to find the distances between $t_i$ and all other vertices, and then, in \Tau, moving upwards from $B_i$ towards the root while checking the distance between $t_i$ and the vertices in each visited bag. Finally, marking all vertices within distance $(2g-i+1)+4tl$ from $c_i$ can be done in a Breadth-First manner in $O(n^2)$ time. In summary, each iteration of \bg takes $O(n^2)$ time, and since there are $O(d)$ iterations, the complexity of \bg is $O(d n^2)$. The algorithm's overall time complexity would be $O(n^2 d \log d) \in O(n^3 \log n)$.
\qed 
\end{proof}

%% file: conclusions.tex
\section{Concluding Remarks}
\input{conclusions0}
\input{conclusions-distributed}

%% file: conclusions0.tex
We introduced algorithms for burning graphs of bounded path-length and tree-length. 
Studying the burning problem with respect to other graph parameters is an interesting direction for future work. In particular, one can consider the problem in graphs of bounded path-width and tree-width. 
As mentioned earlier, the burning problem remains NP-hard for basic graph families such as trees and forests of disjoint paths. While there is a simple algorithm with approximation ratio of at most 3 for general graphs~\cite{BonaKam}, improving this ratio is relatively hard. We provided an algorithm with approximation factor of $2+o(1)$ for burning graphs of bounded tree-length. Providing algorithms with similar guarantees for other graph families 
such as planar graphs is an interesting topic for future research. 

%% file: conclusions-distributed.tex
Another potential direction for future research is to consider distributed algorithms for choosing the activators in a burning sequence. In fact, graph burning can be viewed as a variant of constructing $k$-dominating sets or clustering \cite{KuttenP98}, which are often used in efficient network algorithms and distributed data structures. The activators specified in a burning sequence can be considered the dominators (or cluster centers) but since they are all activated at different times and grow at the same rate, the radius of each cluster is different rather than a fixed value $k$ for all clusters. It would be interesting to investigate potential applications of this kind of network clustering, and to determine how it can be constructed efficiently in a distributed manner.